\documentclass[reqno]{amsart}
\usepackage[latin1]{inputenc}
\usepackage[T1]{fontenc}
\usepackage[english]{babel}
\usepackage{amssymb,amsmath}
\usepackage{amsthm}
\usepackage{mathrsfs}
\usepackage{geometry}

\newtheorem{thm}{Theorem}[section]
\newtheorem{defini}{Definition}[section]
\newtheorem{rem}{Remark}[section]
\newtheorem{lem}{Lemma}[section]
\newtheorem{prop}{Proposition}[section]
\newtheorem{coro}{Corollary}[section]
\newtheorem{ex}{Examples}

\begin{document}
\title[ ]{ Spectral inclusions between $C_0$-quasi-semigroups and their generators }

\author[ A. Tajmouati, Y. Zahouan and M.A. Ould Mohamed Baba
\\]
{ A. Tajmouati,  Y. Zahouan and M.A. Ould Mohamed Baba}
\address{A. Tajmouati, Y. Zahouan and M.A. Ould Mohamed Baba\newline
 Sidi Mohamed Ben Abdellah
 Univeristy,
 Faculty of Sciences Dhar Al Mahraz, Fez,  Morocco.}
\email{abdelaziz.tajmouati@usmba.ac.ma}
\email{zahouanyouness1@gmail.com}
\email{bbaba2012@gmail.com}
\maketitle

\begin{center}
	\textbf{Abstract}
\end{center}
In this paper,  we show a spectral inclusion of a different spectra of a $C_{0}$-quasi-semigroup and its generator and precisely for ordinary, point, approximate point, residual, essential and regular spectra.\\

\textbf{keywords}:
$C_0$-quasi-semigroup, $C_0$-semigroup, semi-regular, ascent, descent, spectrum,  point spectrum,  essential, regular spectra.

\section{\textbf{Introduction and preliminaries}}
Let $X$ be a complex Banach space and $\mathcal{B}(X)$ the algebra of all bounded linear operators on  $X$. We denote by $D(T)$, $Rg(T)$, $Rg^\infty(T):=\cap_{n\geq 1}Rg(T^n)$, $N(T)$, $\rho(T)$, $\sigma(T),$ and $\sigma_p(T)$ respectively the domain, the range, the hyper range, the kernel, the resolvent and the spectrum of $T$, where
$\sigma(T)=\{\lambda\in\mathbb{C}\,\backslash \, \lambda-T \,\mbox{is not bijective}\}$.
The point , the approximate point , the residual and regular spectrum spectra are defined by
\begin{itemize}
 \item $\sigma_p(T)=\{\lambda \in \mathbb{C}\,\backslash \, \lambda -T \,\mbox{is not injective }\}$ 
 \item $\sigma_a(T)=\{\lambda\in  \mathbb{C}\,\backslash \, \lambda-T \,\mbox{is not injective or $Rg(\lambda -T)$ is not closed in  X }\} $. and from  \cite[lemma 1.9]{r.7}  $\lambda$ is an approximate eigenvalue,if and only if there exists a sequence $(x_{n})_{n \in \mathbb{N}} \subset D(A)$ called an approximate eigenvector , such that $\|x_{n} \| =1 $ and $lim_{n\to \infty}\| Ax_{n} - \lambda x_{n} \| = 0$.
 \item $\sigma_r(T)=\{\lambda\in  \mathbb{C}\,\backslash \,  \,\mbox{ $Rg(\lambda -T)$ is not dense in  X }\} $
 \item $\sigma_\gamma(T)=\{\lambda \in \mathbb{C}\,\backslash \, \lambda -T \,\mbox{is not semi regular }\}$, ie $\lambda \in \sigma_\gamma(T) $ if $Rg(\lambda -T)$ is not closed or $N(\lambda -T) \not\subset Rg^\infty(\lambda -T)$
\end{itemize}

An operator $T\in \mathcal{B}(X)$ is called Fredholm operator, in symbol $T\in \Phi(X)$, if $\alpha(T)=dim N(T)$ and $\beta(T)=codim Rg(T)$ are finite, and the essential  spectrum is defined by,
\begin{center}	
	$\sigma_e(T)=\{\lambda\in\mathbb{C}\,:\, \lambda-T \notin \Phi(X)\}.$
\end{center}

The family $(T(t))_{t\geq 0}\subseteq \mathcal{B}(X)$ is a $C_0$-semigroup if it has the following properties:

\begin{enumerate}
	\item $T(0)=I$
	\item $T(t)T(s)=T(t+s)$;
	\item The map $t\rightarrow T(t)x$ from $[0,+\infty[$ into $X$ is continuous for all $x\in X$;	
\end{enumerate}

In this case, its generator $A$ is defined by
$$\mathcal{D(A)}=\{x\in X\, / \lim_{t\rightarrow 0^+}\frac{T(t)x-x}{t} \; exists \},$$
with $$Ax= \lim_{t\rightarrow 0^+}\frac{T(t)x-x}{t}.$$

%%%%%%%%%%%%%%%%%%%%%%%%%%%%%%%%%%%%%%%%%%%%%%%%%%%%%%%%%%%%%%%%%%%%%%%%%%%%%%%%%%%%%%%%%%%%%%%%%ù

The theory of quasi-semigroups of  bounded linear operators, as a generalization of semigroups of operators, was introduced by Leiva and Barcenas \cite{r.3} , \cite{r.4} ,\cite{r.5}.

\begin {defini}\cite{r.3}
Let $X$ be a complex Banach. The family $\left\lbrace R(t,s)\right\rbrace _{t,s\geq 0} \subseteq \mathcal{B}(X)$   is called  a  strongly  continuous  quasi-semigroup   (or   $C_{0}$-quasi-semigroup)  of  operators  if  for  every  $t,s,r\geq 0$  and   $x \in X$  ,
\begin{enumerate}
	\item $R(t,0)= I $, the identity operator on  $X$,
	\item $R(t,s+r)=R(t+r,s)R(t,r)$,
	\item $lim_{s\longrightarrow 0}\left| \left|R(t,s)x - x \right| \right| = 0$ , 
	\item There exists a continuous increasing mapping $ M : [ 0 ; +\infty [ \longrightarrow [ 1 ; +\infty [ $ such that,
	\begin{center}
		$\left| \left|R(t,s)\right| \right| \leq  M(t+s)$
	\end{center}
	
\end{enumerate}
\end {defini}

\begin{defini}\cite{r.3}
	For a  $C_{0}$-quasi-semigroup  $\left\lbrace R(t,s)\right\rbrace _{t,s\geq 0}$ on a Banach space   $X$, let $\mathcal{D}$ be the set of all   $x \in X$    for which 
	the following limits exist,
	\begin{center}
		$ \lim_{s\rightarrow 0^+}\frac{R(0,s)x-x}{s}$ \; \; and  \; \; $ \lim_{s\rightarrow 0^+}\frac{R(t,s)x-x}{s} =  \lim_{s\rightarrow 0^+}\frac{R(t-s,s)x-x}{s}$ , \; $ t> 0 $
	\end{center}
	
	For  $t\geq 0$    we define an operator $A(t)$  on $\mathcal{D}$ as  $A(t)x = \lim_{s\rightarrow 0^+}\frac{R(t,s)x-x}{s}$
	
	The family  $\left\lbrace A(t)\right\rbrace _{t\geq 0}$ is called infinitesimal generator of the $C_{0}$-quasi-semigroups $\left\lbrace R(t,s)\right\rbrace _{t,s\geq 0}$.\\
	
	Throughout this paper we  denote 
	$T(t)$  and  $R(t,s)$ as $C_{0}$-semigroups  $\left\lbrace T(t)\right\rbrace _{t\geq 0}$ and  $C_{0}$-quasi-semigroup  $\left\lbrace R(t,s)\right\rbrace _{t,s\geq 0}$  respectively.We also denote $\mathcal{D}$ as domain for  $A(t)$ , $t\geq 0$.
	\end {defini}
	\begin{rem}\cite [Examples  2.3 and 3.3]{r.12}
	In the semigroups theory, if   $A$  is an infinitesimal generator of $C_{0}$-semigroup with domain   $\mathcal{D(A)}$  , then  $A$      
	is a closed operator  and $\mathcal{D(A)}$ is dense in $X$.  These are is not always true for any  $C_{0}$-quasi-semigroups.
		\end{rem}
	\begin{rem}
			If $ T(t) $ be  a  $C_{0}$-semigroup   on a Banach space   $X$ with its generator $A$ then $R(t,s)$ with $R(t,s) =T(s),$  $t,s\geq 0$ defines a $C_{0}$-quasi-semigroup on  $X$ with  generator  $A(t) = A$ , $t\geq 0 $ and $\mathcal{D} = \mathcal{D(A)}$ .
	\end{rem}
	%%%%%%%%%%%%%%%%%%%%%%%%%%%%%%%%%%%%%%%% example %%%%%%%%%%%%%%%%%%%%%%%%%%%%%%%%%%%%%
	\begin{ex}
		
		Let $ T(t) $ be  a  $C_{0}$-semigroup   on a Banach space   $X$ with its generator $A$.\\	
		For $t,s\geq 0$ , $R(t,s) = T(g(t+s) - g(t)) $,	where $  g(t) = \int_{0}^t  a(u)du $ and $a \in \mathcal{C}([0 ; \infty[)$ with $a(t) > 0$
		
		The $R(t,s)$ is $C_{0}$-quasi-semigroup on  $X$ with  generator  $A(t) = a(t)A$.
	\end{ex}
	
	\begin{proof}:
  
		\begin{enumerate}
			\item $R(t,0)=T(g(t) - g(t))=T(0)= I $.
			\item \begin{align*} R(t,s+r) &=T(g(t+s+r) -g(t+r)+ g(t+r) - g(t))\\
        &=T(g(t+s+r) -g(t+r) T( g(t+r) - g(t)) \\
        &=R(t+r,s)R(t,r),
        	\end{align*}
			\item $lim_{s\longrightarrow 0}\left| \left|R(t,s)x - x \right| \right| = lim_{s\longrightarrow 0}\left| \left|T(g(t+s) - g(t))x - x \right| \right|= 0$ , since $ g$ is continuous.
			\item  Since  $ T(t) $   is strongly continuous on   $X$  , there exists  $\omega$  and $M_{\omega} > 0 $ such that ,
			\begin{center}
				$\left| \left|T(t)  \right| \right|   \leq  M_{\omega}e^{\omega t}$
			\end{center}
			
			Therefore, $\left| \left|R(t,s)\right| \right| \leq M(t+s)$ , where $M(t+s) =  M_{\omega}e^{\omega( (g(t+s)-g(t))} $
			
			Moreover, 
			\begin{align*}
			A(t)x &= \lim_{s\rightarrow 0^+}\frac{R(t,s)x-x}{s}\\
			&= \lim_{s\rightarrow 0^+}\frac{T(g(t+s) - g(t))x-x}{s}\\
			&= g^{'}(t+s)\dfrac{d}{ds}\left[T(g(t+s) - g(t))x \right] | _{s=0} \\ 
			&= a(t)Ax 
			\end{align*}	
		\end{enumerate}
		Thus, $R(t,s)$ is $C_{0}$-quasi-semigroup on  $X$ with  generator  $A(t) = a(t)A$.
	\end{proof}

	%%%%%%%%%%%%%%%%%%%%%%%%%%%%%%%%%%%%%%%%%%% theorem %%%%%%%%%%%%%%%%%%%%%%%%%%%
	The following results are obtained recently by Sutrima , Ch. Rini Indrati and others \cite{r.12} , there are show some relations between a $C_{0}$-quasi-semigroup and its  generator.
	
	\begin{thm}\cite{r.12}
		Let $R(t,s)$ be a $C_{0}$-quasi-semigroup on  $X$ with  generator  $A(t)$ then,
		
		\begin{enumerate}
			\item For each   $t\geq 0$  ,  $R(t,.)$ is strongly continuous on  $ [ 0 ; +\infty [ $.
			\item For each   $t\geq 0$ and  $x\in X$,
			\begin{center}
				$	\lim_{s\rightarrow 0^+} \frac{1}{s}\int_{0}^{s}R(t,h)x dh = x$
			\end{center}
			\item If $x \in \mathcal{D} $ ,  $t\geq 0$ and $t_{0},s_{0}\geq 0$ then,  $R(t_{0},s_{0})x \in \mathcal{D} $ and 
			\begin{center}
				$R(t_{0},s_{0})A(t)x = A(t)R(t_{0},s_{0})x $
			\end{center}
			\item For each $s > 0$ , $\frac{\partial}{\partial s} (R(t,s)x)= A(t+s)R(t,s)x = R(t,s) A(t+s)x $;  $ x \in \mathcal{D}. $
			\item If  $A(.)$ is locally integrable, then for every $x \in \mathcal{D} $ and   $s \geq 0$,
			
			\begin{center}
				$	R(t,s)x = x + \int_{0}^{s} A(t+h)R(t,h)x dh .  $
			\end{center}
			\item If $ f : [ 0 ; +\infty [ \longrightarrow X$  is a continuous, then for every $t \in [ 0 ; +\infty [$  
			
			\begin{center}
				$	\lim_{r\rightarrow 0^+} \int_{s}^{s+r}R(t,h) f(h)x dh = R(t,s)f(s)$
			\end{center}
		\end{enumerate}
		
	\end{thm}
	
	In this work, we show that the spectral inclusion of different spectra  of $C_{0}$-semigroups  valid for $C_{0}$-quasi-semigroups.

%%%%%%%%%%%%%%%%%%%%%%%%%%%%%%%%  Main  results	%%%%%%%%%ù
	\section{\textbf{Main results}}
		For later use, we introduce the following operator acting on $X$ and depending on the parameters $\lambda \in \mathbb{C}$ and  $ t,s\geq 0$ :
	\begin{center}
		$D_\lambda(t,s)x = \int_0^s e^{\lambda(s-h)}R(t,h)xdh$ \; \;for all $x\in X$ .
	\end{center}
	$D_\lambda(t,s)$ is a bounded linear operator on $X$\\
	
	\begin{thm} \label{t1} Let $A(t) $ be the generator of the C0-quasi-semigroup  $\left\lbrace R(t,s)\right\rbrace _{t,s\geq 0}$ . Then for all $\lambda\in\mathbb{C}$ and all $t,s\geq 0$ we have
		\begin{enumerate}
			\item For all $x\in X$, $$(\lambda - A(t))D_\lambda(t,s)x=[e^{\lambda s} -R(t,s)]x .$$
			\item For all $x\in \mathcal{D} $, $$D_\lambda(t,s)(\lambda - A(t))x=[e^{\lambda s} - R(t,s)]x .$$
		\end{enumerate}
	
	\end{thm}
	
	%%%%%%%%%%%%%%%%%%%%%%%%%%%%%%%%%%%%%%%%%
	\begin{proof}
		\begin{enumerate}
			\item For all $x\in X$ we have
			
			\begin{eqnarray*}
				R(0,r)D_\lambda(t,s)x   &=&   R(0,r) \int_0^s e^{\lambda(s-h)}R(t,h)xdh\\
				&=&\int_0^s e^{\lambda(s-h)}R(0,r)R(t,h)xdh \\
			\end{eqnarray*}

			And  we obtain ,
			
			\begin{eqnarray*}
				\lim_{r\rightarrow 0^+}\frac{R(0,r)D_\lambda(t,s)x - D_\lambda(t,s)x}{r} &=& \lim_{r\rightarrow 0^+}\frac{\int_0^s e^{\lambda(s-h)}R(0,r)R(t,h)xdh -\int_0^s e^{\lambda(s-h)}R(t,h)xdh}{r}\\
				&=&\frac{\partial}{\partial r}\left[ \int_0^{s}e^{\lambda(s-h)}R(0,r)R(t,h)xdh \right] _{|r=0} \\
			\end{eqnarray*}

			Then  $\lim_{r\rightarrow 0^+}\frac{R(0,r)D_\lambda(t,s)x - D_\lambda(t,s)x}{r} $   exists.\\

			And,
			
			\begin{eqnarray*}
				\lim_{r\rightarrow 0^+}\frac{R(t,r)D_\lambda(t,s)x - D_\lambda(t,s)x}{r}  &=&   \lim_{r\rightarrow 0^+}\frac{\int_0^s e^{\lambda(s-h)}R(t,r)R(t,h)xdh -\int_0^s e^{\lambda(s-h)}R(t,h)xdh}{r}\\
				&=& \frac{\partial}{\partial r}\left[ \int_0^{s}e^{\lambda(s-h)}R(t,r)R(t,h)xdh \right] _{|r=0} \\
				&=&\left[ \int_0^{s}e^{\lambda(s-h)}\frac{\partial}{\partial r}(R(t,r))R(t,h)xdh \right] _{|r=0} \\
				&=& \left[ \int_0^{s}e^{\lambda(s-h)}A(t+r)R(t,r)R(t,h)xdh \right] _{|r=0} \\
				&=& \int_0^{s}e^{\lambda(s-h)}A(t)R(t,h)xdh      
			\end{eqnarray*}

			Moreover,
			\begin{eqnarray*} 
				\lim_{r\rightarrow 0^+}\frac{R(t - r,r)D_\lambda(t,s)x - D_\lambda(t,s)x}{r}  &=&    \lim_{r\rightarrow 0^+}\frac{\int_0^s e^{\lambda(s-h)}R(t-r,r)R(t,h)xdh -\int_0^s e^{\lambda(s-h)}R(t,h)xdh}{r}\\
				&=& \frac{\partial}{\partial r}\left[ \int_0^{s}e^{\lambda(s-h)}R(t-r,r)R(t,h)xdh \right] _{|r=0} \\
				&=&\left[ \int_0^{s}e^{\lambda(s-h)}\frac{\partial}{\partial r}(R(t-r,r))R(t,h)xdh \right] _{|r=0} \\
				&=&\left[ \int_0^{s}e^{\lambda(s-h)}A(t)R(t-r,r)R(t,h)xdh \right] _{|r=0} \\
				&=&\int_0^{s}e^{\lambda(s-h)}A(t)R(t,h)xdh     
			\end{eqnarray*}

			Thus , $\lim_{r\rightarrow 0^+}\frac{R(t,r)D_\lambda(t,s)x - D_\lambda(t,s)x}{r}  = \lim_{r\rightarrow 0^+}\frac{R(t - r,r)D_\lambda(t,s)x - D_\lambda(t,s)x}{r}$

			Hence, we deduce that $D_\lambda(t , s)x\in \mathcal{D}$ And , \\
			
			\begin{eqnarray*}
				A(t)D_\lambda(t,s)x  &=& \int_0^{s}e^{\lambda(s-h)}A(t)R(t,h)xdh \\
				&=&  \int_0^{s}e^{\lambda(s-h)}A(t+h)R(t,h)xdh ,\; \; because \; \;  A(t+h) = A(t) for \; all \; \;   t,h \geq 0 .\\
				&=&  \int_0^{s}e^{\lambda(s-h)}\frac{\partial}{\partial h}(R(t,h))xdh \\
				&=&  \left[ e^{\lambda(s-h)}R(t,h) \right]_{0}^{s} + \lambda\int_0^{s}e^{\lambda(s-h)}R(t,h)xdh \\
				&=&  R(t,s)x - e^{\lambda s}x + \lambda D_\lambda(t,s)x
			\end{eqnarray*}                    
			
			Finaly , $(\lambda - A(t))D_\lambda(t,s)x=[e^{\lambda s} -R(t,s)]x $ for all $x \in X $.

			\item For all $x\in \mathcal{D} $ and all $t ,s \geq 0$ we have ,
			\begin{eqnarray*}
				D_\lambda(t ,s)A(t)x &=& \int_0^s e^{\lambda(s-h)}R(t,h)A(t)xdh\\
				&=& \int_0^s e^{\lambda(s-h)}R(t,h)A(t+h)xdh\\
				&=& \int_0^{s}e^{\lambda(s-h)}\frac{\partial}{\partial r}(R(t,h))xdh\\
				&=&  \left[ e^{\lambda(s-h)}R(t,h) \right]_{0}^{s} + \lambda\int_0^{s}e^{\lambda(s-h)}R(t,h)xdh \\
				&=&  R(t,s)x - e^{\lambda s}x + \lambda D_\lambda(t,s)x
			\end{eqnarray*}
			Thus, we deduce for all $x\in D(A)$
			$D_\lambda(t,s)(\lambda - A(t))x=[e^{\lambda s} - R(t,s)]x $.
		\end{enumerate}
	\end{proof}
	
	\begin{coro}\label{c1}
		In the case of $C_0$-semigroup $T(s)= R(t,s)$ , we retrieve the equality \cite{r.9},
		\begin{center}
			$(\lambda-A)D_\lambda(s)x=[e^{\lambda s}-T(s)]x $
		\end{center}
		With $D_\lambda(s)x = \int_0^s e^{\lambda(s-h)}T(h)xdh$ for all $x\in X$  and $s\geq 0$
	\end{coro}
	
	\begin{coro}\label{c2}  Let $A(t)$ be the generator of a $C_{0}$-quasi-semigroup $(R(t,s))_{t,s\geq 0}$ . Then for all $\lambda \in \mathbb{C}$ , $t,s\geq 0$ and $n \in \mathbb{N}$ ,
		\begin{enumerate}
			\item For all  $x\in X$ ,
			\begin{center}
				$(\lambda-A(t))^n[D_\lambda(t,s)]^nx=[e^{\lambda s}-R(t,s)]^nx$.
			\end{center}
			\item For all  $x\in \mathcal{D}^n$ \, (Domain of $A(t)^n$) ,  
			\begin{center}
				$[D_\lambda(t,s)]^n(\lambda-A(t)^n)x=[e^{\lambda s}-R(t,s)]^nx.$
			\end{center}
			\item $N[\lambda-A(t)]\subseteq N[e^{\lambda s}-R(t,s)].$
			\item $Rg[e^{\lambda s}-R(t,s)]\subseteq Rg[\lambda-A(t)].$
			\item $N[\lambda-A(t)]^n\subseteq N[e^{\lambda s}-R(t,s)]^n.$
			\item $Rg[e^{\lambda s}-R(t,s)]^n\subseteq Rg[\lambda-A(t)]^n.$
			\item $Rg^\infty[e^{\lambda s} - R(t,s)]\subseteq Rg^\infty[\lambda-A(t)].$
		\end{enumerate}
	\end{coro}
	
	\begin{proof}
		
		follow easily from ,
		\begin{eqnarray*}
			e^{\lambda s}x - R(t,s)x &=& (\lambda - A(t))D_\lambda(t,s)x \\
			&=&   D_\lambda(t,s)(\lambda - A(t))x
		\end{eqnarray*}
		\end {proof}
		%%%%%%%%%%%%%%%%%%%%%%%%%%%%%%%%%%%%%%%%%%%%%%%%%%%%%%%%%
		The following theorem characterizes the  ordinary, point, approximate point, essential and residual spectra  of a $C_{0}$-quasi-semigroup.
		
		\begin{thm} For the generator $A(t)$ of a  $C_{0}$-quasi-semigroup $(R(t , s))_{t , s\geq 0 }$ there exist the spectral inclusions
			\begin{enumerate}
				\item $ e^{\sigma(A(t))s}\subset \sigma(R(t,s))$
				\item $ e^{\sigma_p(A(t))s}\subset \sigma_p(R(t,s))$
				\item $ e^{\sigma_a(A(t))s}\subset \sigma_a(R(t,s)).$
				\item $ e^{\sigma_e(A(t))s}\subset \sigma_e(R(t,s)).$
				\item $ e^{\sigma_r(A(t))s}\subset \sigma_r(R(t,s)).$

			\end{enumerate}
		\end{thm}
		
		%%%%%%%%%%%%%% PROOOF
		\begin{proof}
			$\,$\\
			\begin{enumerate}
				\item Let $\lambda\in \mathbb{C}$ such that for all $t\geq0$
				$$\ e^{\lambda s}\notin \sigma(R(t,s)),$$
				then the operator $ e^{\lambda s}-R(t,s)$ is invertible where $F_\lambda(t,s)$ its inverse.\\
				By Theorem \ref{t1}, we obtain for every $x\in \mathcal{D}$
				\begin{eqnarray*}
					x &=& F_\lambda(t,s)[e^{\lambda s} - R(t,s)]x\\
					&=& F_\lambda(t,s)[D_\lambda(t,s)(\lambda -A(t))]x\\
					&=& [F_\lambda(t,s)D_\lambda(t,s)](\lambda -A(t))x.
				\end{eqnarray*}
				On the other hand, also from Theorem \ref{t1}, we obtain for every $x\in X$
				\begin{eqnarray*}
					x &=& [e^{\lambda s}- R(t,s)]F_\lambda(t,s)x;\\
					&=& [(\lambda-A(t))D_\lambda(s)]F_\lambda(t,s)x;\\
					&=& (\lambda-A(t))[D_\lambda(t,s)F_\lambda(t,s)]x.
				\end{eqnarray*}
				Since we know that $R(t,s)F_\lambda(t,s)=F_\lambda(t,s)R(t,s)$, then
				$$F_\lambda(t,s)D_\lambda(t,s) = D_\lambda(t,s)F_\lambda(t,s).$$
				Finally, we conclude that $\lambda -A(t)$ is invertible and hence $\lambda\notin\sigma(A(t)).$

				\item Let $\lambda\in\sigma_p(A(t))$, then there exists $x\neq 0$ such that $x\in N(\lambda-A(t)).$
				From Corollary \ref{c2}, we get $x\in N[e^{\lambda s}-R(t,s)].$
				Therefore, we conclude that  $ e^{\lambda t}\in \sigma_p(R(t,s))$.
				
				\item  Let $\lambda\in\sigma_{a}(A(t))$ and a corresponding approximate eigenvector  $(x_{n})_{n \in \mathbb{N}} \subset \mathcal{D}$ , we define the sequence $(y_{n})$ by $y_{n}: = e^{\lambda s}x_{n} - R(t,s)x_{n} $\\
				By theorem \ref{t1} we have
				\begin{center}
					 $y_{n}: =\int_0^s e^{\lambda(s-h)}R(t,h)(\lambda - A(t))x_{n}dh. $
				\end{center}
			So there is a constant $c \; \rangle \;  0$ such that ,
		\begin{center}
				$\|y_{n}\| =  \int_0^s \|e^{\lambda(s-h)}R(t,h)(\lambda - A(t))x_{n}\|dh \leq c\|(\lambda - A(t))x_{n}\|\to 0 $ as $n\to \infty$	
		\end{center}
	Hence, $ e^{\lambda s}$ is an approximate eigenvalue of $R(t,s)$, and $(x_{n})_{n \in \mathbb{N}}$ serves as the same approximate eigenvector for all $t,s \geq 0$

		\item Let $\lambda\in \mathbb{C}$ such that $$ e^{\lambda t}\notin \sigma_e(R(t,s))).$$
				Then we have $\alpha[e^{\lambda t}-R(t,s)]<+\infty $ \, and \, $
				\beta[e^{\lambda t} - R(t,s)]<+\infty.$
				Therefore, by Corollary \ref{c2}, we conclude that
				$\alpha[\lambda-A(t)]<+\infty$ \, and \, $\beta[\lambda-A(t)]<+\infty,$
				and hence
				$\lambda\notin \sigma_e(A)$.
				
			\item 	Let  $\lambda\in\sigma_{r}(A(t))$,then $ Rg[\lambda-A(t)] $ is not dense , now we use the corollary \ref{c2} we obtain  
			\begin{center}
				$Rg[e^{\lambda s}-R(t,s)]\subseteq Rg[\lambda-A(t)].$
			\end{center}
			Thus $Rg[e^{\lambda s}-R(t,s)]$ is not dense and finally  $e^{\lambda s} \in \sigma_r(R(t,s)).$
				
			\end{enumerate}
		\end{proof}
		
		%%%%%%%%%%%%%%%%%%%%%%%%%%%%%% regular spectrum %%%%%%%%%%%
		In the next theorem, we will prove that the spectral inclusion of $C_{0}$-quasi-semigroups remains true for the regular spectrum.
			\begin{thm} 
			For the generator $(A(t),\mathcal{D})$  of a $C_{0}$-quasi-semigroup $(R(t ;s))_{t , s\geq 0 }$ on a Banach space $X$ ,we have the inclusion :
			
			\begin{center}
				 $ e^{\sigma_{\gamma}(A(t))s}\subset \sigma_{\gamma}(R(t,s))$
			\end{center}
			\end{thm}
	To prove this result, we need the following proposition and lemma.
			\begin{prop} \cite[Corrolary 1.5]{r.9}
				Let $T$ be a closed operator.\\
				If $T$ is semi regular then $Rg^\infty(T)$ is closed
			\end{prop} 
		
			\begin{lem}\cite[Lemma 1]{r.8} \label{l1}
		Let $T \in B(X)$, If $T$ is semi-regular , then the operator
    $$ \hat{T} : X/Rg^\infty(T) \to X/Rg^\infty(T) $$
		 induced by T is bounded below.
		\end{lem}
	%%%%%%%%%%%%%%%%%%%%  preuve of the theorem %%%%%%%%%%%%%
	
		\begin{proof}
	Let $\lambda\in \mathbb{C}$ and $ s_{0} > 0$ be fixed such that $ e^{\lambda s_{0}} \notin \sigma_\gamma(R(t,s_{0}))$ for all $t\geq 0,$ then $e^{\lambda s_{0}}-R(t,s_{0})$ is semi-regular. We show that $\lambda - A(t) $ is  semi-regular.\\
 For this, consider the closed 
  $(R(t,s))_{t,s \geq 0}$-invariant subspace $ M := Rg^\infty(e^{\lambda s_{0}} - R(t,s_{0}))$ of X and the quotient $C_{0}$-quasi-semigroup $(\widehat{R(t,s)})_{t,s\geq 0}$ defined on $X/M$ by
  $$\widehat{R(t,s)}\widehat{x} := \widehat{R(t,s)x} ,\; \;  for \; \; \widehat{x} \in X/M $$
with generator $\widehat{A(t)}$ defined by
	$$\mathcal{\widehat{D}}:= \{ \widehat{x} ,\;  x\in \mathcal{D} \},\; \; \; \widehat{A(t)}\widehat{x} := \widehat{A(t)x} ,\; \;  for \; \; \widehat{x} \in  \mathcal{D}$$
	
	From Lemma \ref{l1} , it follows that the operator $e^{\lambda s_{0}}-R(t,s_{0})$ is bounded below. Thus, $ e^{\lambda } \notin \sigma_a(R(t,s_{0})) $ and from theorem 2.2 (3) $ \lambda \notin \sigma_a(\widehat{A(t)}) $.\\
	In consequence, the operator $ \lambda - \widehat{A(t)} $ is injective and the closed range and from corollary 2.2 (7) 
	Then , $$ N(\lambda - A(t)) \subset Rg^\infty((\lambda - A(t)) $$
	
	Now, we show that $Rg(\lambda - A(t))$is closed.\\
	To do this, consider a sequence $(u_{n})_{n \geq 0}$ of elements of $Rg(\lambda - A(t))$ , which converges to u. Then, there exists a sequence $(v_{n})_{n \geq 0}$ of elements of $\mathcal{D}$ such that $(\lambda - A(t))u_{n} = v_{n} \to u $. Since $Rg(\lambda - \widehat{A(t)})$ is closed, there exists $\widehat{w} \in \widehat{\mathcal{D}}$ such that $\widehat{u}= (\lambda - A(t))\widehat{w}$.\\
	Hence, $u - (\lambda - A(t))w \in Rg^\infty(e^{\lambda s_{0}}-R(t,s_{0})) \subset  Rg^\infty(\lambda - A(t)) \subset Rg(\lambda - A(t)) $.
	Therefore, $u \in Rg(\lambda - A(t)) $\\
	Consequently, the operator $\lambda  - A(t)$  is  semi-regular.
	\end{proof}
	
	In the next work we will try to demonstrate the equality of the spectra or give counter-examples in the case of a strict inclusion.

	%%%%%%%%%%%%%%%%%%%%% BIBLIO
	{
		}

	\end{document}